\documentclass[12pt]{elsarticle}
\usepackage[colorlinks=false, pdfborder={0 0 0}]{hyperref}
\usepackage{graphicx}
\usepackage{float}
\usepackage{subcaption}
\usepackage{algorithm}
\usepackage{algpseudocode}
\usepackage{doi}

\usepackage{amsmath}
\usepackage{amssymb}
\newtheorem{corollary}{Corollary}
\newtheorem{example}{Example} 
\newtheorem{lemma}{Lemma}
\newtheorem{proposition}{Proposition}
\newtheorem{theorem}{Theorem}
\newenvironment{proof}[1][Proof]  {\begin{trivlist}
\item[\hskip \labelsep {\bfseries #1}]} {\end{trivlist}}

\begin{document}

\begin{frontmatter}

\title{A novel monotonous finite volume element scheme for convection dominant diffusion problem\tnoteref{label1}}

\author{\textbf{Cunyun Nie$^{1,*}$, Xiaoling Chen$^{2,*}$, Zhujun wang$^{1}$, Zhikun Tian$^{1}$, Chengjie Xia$^{1}$}}
\address{$1.$ \textsl{{School of Computational Science and Electronics, Hunan Institute of Engineering, 411104, Hunan, China}}}
\address{$2.$ \textsl{ School of Future Technology, Hunan Institute of Engineering, 411104, Hunan, China}}

 \tnotetext[S]{ The work is supported in part by Hunan Provincial Department of Educational Key Project (No.25A0517), Hunan Provincial Social Science Research Key Project‌(No. 25ZDB031), Hunan Provincial Natural Science Foundation (No.2026JJ70070, No.2026JJ80145).\\
 $*$ Corresponding author: Cunyun Nie, E-mail address: ncy1028@163.com.}


\begin{abstract}
One novel monotonous finite volume element (MFVE) scheme is put forward for convection dominant diffusion problem.  The main contributions of this paper include four aspects. Firstly, one upwind sub-control volume, called upwind volume, is introduced for the discretization of the convection term similar to the upwind element, which leads to the upwind property. Secondly, one first-order linear discrete operator for the convection term in the balance equation is directly obtained by numerical integration in upwind volume, which is different from the classical finite volume methods.Thirdly, one asymptotic expansion of gradient functions in the dual element of each node is derived by one extension interpolation function. Then, one second-order nonlinear discrete operator for the convection term is constructed by the asymptotic expansion and error estimation of its FVE solution, where some perturbed coefficients are skillfully put into the expansion, which serves as a high-order correction. Fourthly, one novel two-order MFVE scheme, also accompanied by another one-order MFVE, is designed for convection-dominated diffusion problems together with one positivity-preserving finite volume element (PFVE) discrete diffusive operator. The L2 norm of the error of the approximate solution is derived. Finally, numerical results validate theoretical conclusions.
\end{abstract}

\begin{graphicalabstract}
\includegraphics[width=.85\textwidth]{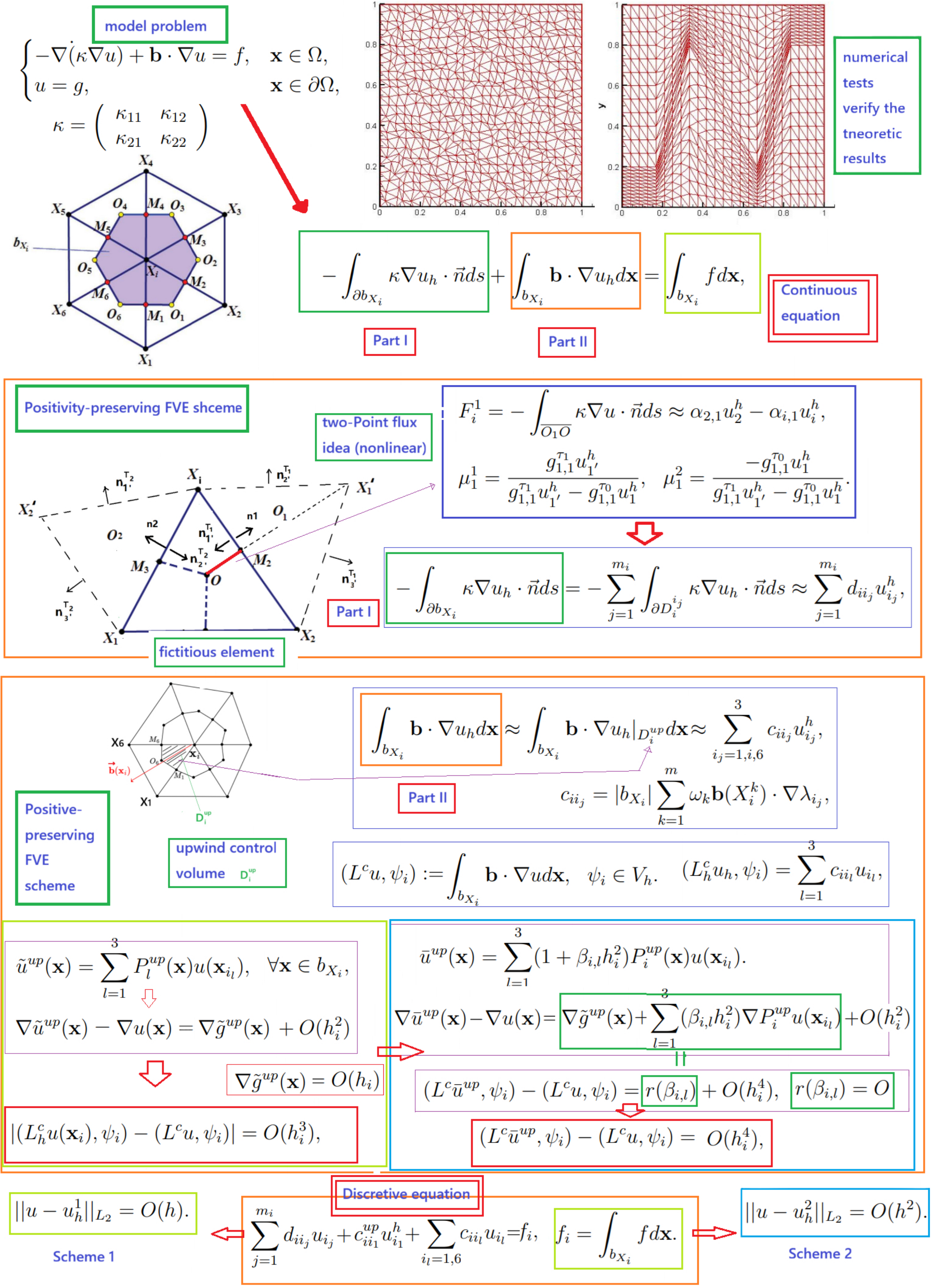}
\end{graphicalabstract}

\begin{highlights}
\item One upwind sub-control volume is initially introduced for the discretization of the convection term.

\item One first-order linear discrete operator for the convection term is directly obtained by numerical integration in upwind volume, which is different from the classical finite volume methods.
    
\item One second-order nonlinear discrete operator for the convection term is constructed by the asymptotic expansion and error estimation of its FVE solution. 
    
\item One novel two-order FVE scheme is designed for convection-dominated diffusion problems, with the L2 norm of the error of the approximate solution  derived.

\end{highlights}



\end{frontmatter}

\section{Intoduction}
\label{aas}




Convection diffusion equation describes the diffusion motion of a physical quantity (such as concentration, temperature, etc.) in a fluid under the action of transport, and is widely applied in numerous fields including fluid mechanics, reservoir simulation, weather forecasting, groundwater prediction and environmental protection \cite{2,3}.
The convection term dominates as the Peclet number approaches large, which leads to its numerical difficulty. There have been many numerical methods for this kind of problems over these years: finite element (FE) method \cite{a-3}, finite difference method \cite{a-1}, finite volume (FV) method \cite{a-2}, spectral method \cite{a-4}, and so on.  The most popular FE method often comes across numerical instability with spurious numerical oscillations, which often brings non-physical solutions \cite{PG-1,DG-1}.  
Many stabilization techniques  \cite{z-1, Tabata} are summarized on the discretization of convection term: the streamline upwind Petrov Galerkin, the free-bubble function method, the local projection method, the edge-averaged FEM, the internal penalty function method and the exponential fitting scheme. 
Some extreme-preserving FE schemes  \cite{z-2,PG-2,EAFE-3,AFC-3} are also constructed to eliminate the non-physical solutions. 

The finite volume method is very popular due to its locally 
preserving conservation property \cite{Lipn,Potier,Peng}.
There are many FV researches on the discretization of convection diffusion problem,
also encountering the similar issues, such as spurious numerical oscillations. In early stage, the upwind finite volume method is put forward, and the discretized convective fluxes are divided into two parts: inflow and outflow parts, according to the dot product of flux velocity and outer normal direction along the boundary of control volume \cite{LiR, Lzzarov,Chen}.
The second approximation of convection flux is based on the two-term Taylor expansion of the approximation in the line integral,  such as the upwind second-order convergent with proper slope limiter, the upstream element selected leading to more than nine points in the scheme, and so on \cite{Sheng2, LiAng}.
The third method is based on the modified dual element which is associated with the local Peclet numbers, and the construction of dual element is determined by the convection and diffusion coefficients, also by the mesh size \cite{Gaoyulong}.
Another method is that the general convection dominated diffusion equation is transformed into a conservative form by using exponential transformation, and a stable and high-precision FVE scheme is constructed \cite{ZouX,ZhangX}.
In addition, there appears one hybrid method for it, the diffusion term with the standard FE scheme, and the convection and source terms with the weighted upwind FV scheme, which overcomes numerical oscillation, avoids numerical dispersion, and has high-order accuracy \cite{SunM}. 
Some other techniques are also presented:  correction technique \cite{Yuan}, high-order convergence \cite{Kroner},  two-grid method \cite{HeM}. The above methods are effective, yet the algorithm are always somewhat complex. Inspired by the idea of upwind element, we will attempt to construct upwind volume and adopt one new approach to discretize the convection term in the balance equation. 

In this paper, one novel MFVE scheme is put forward for convection dominant diffusion problem. Firstly, one upwind volume is put forward to discretize the convection term. Then, the convection term in the balance equation is discretized by numerical integration in upwind volume.
One asymptotic expansion of gradient functions over the dual element of each node is derived. Hence, one second-order nonlinear FVE discrete operator for the convection term is constructed, where some perturbed coefficients are skillfully introduced into this asymptotic expansion. Furthermore, the MFVE scheme is designed for convection dominated diffusion problems together with one positivity-preserving FVE discrete diffusive operator. 

The rest of this paper is organized as follows. Section 2 gives the model problem. Section 3 displays the discretization of diffusion term,  Section 4 constructs the discretization of convection term, and obtains the novel MFVE scheme. Numerical experiments are presented in Section 5. Finally, some conclusions are listed.

\section{Model problem and Preliminary knowledge}
\setcounter{equation}{0}

Consider the model problem for unknown $u=u(\mathbf{x})$, $\mathbf{x}=(x,y)$:
\begin{equation}\label{eq-PDE-1}
\begin{cases}
  - \nabla \dot (\kappa \nabla u) + \mathbf{b} \cdot \nabla u = f, &\mathbf{x}\in \Omega, \\
  u = g, &\mathbf{x}\in \partial \Omega,
\end{cases}
\end{equation}
where $\Omega \subset \mathbb{R}^2$, 
$\mathbf{b} \in W^{1,\infty}(\Omega)^d, f\in L^2(\Omega), g \in H^{1/2}(\partial\Omega)$ and $$
\kappa=\left(
         \begin{array}{cc}
           \kappa_{11} & \kappa_{12} \\
           \kappa_{21} & \kappa_{22} \\
         \end{array}
       \right)
$$
is the diffusion tensor which is real-value and symmetric and strictly positive definite. \\

Let $\Omega_h=\{E_k, 1\leq k \leq M  \}$ be the triangular partition of
Region $\Omega$, and $\Omega_h^{*} = \{ b_{X_i} , 1 \leq i \leq N \}$ be the corresponding dual partition, and
$b_{X_i}$ be the dual element (also called the control volume) about Node $X_i$ (Shown as Fig. \ref{kzt}~(a)),
where $M$ and $N$ are the numbers of total elements and partition nodes, respectively, and $h$ denotes the maximum diameter of all elements in the partition $\Omega_h$.
In Fig. \ref{kzt}~(a), Point $M_k$ is the midpoint of $X_iX_k, 1\leq k \leq 6$, and $O_k$ is the barycenter of $\triangle X_iX_kX_{k+1}$, and $X_7=X_1$. For any element $E_k$ (also denoted as $\tau_k$), $D_i, i=1,2,3$ (shown as Fig. \ref{kzt}~(b)) is the sub-control volume about Node $X_i$, respectively.
The dual element $b_{X_i}$ can be seen as the union of several sub-control volumes, i.e.,  $b_{X_i}=\cup_{k=1}^6 D_i^k$, where $D_i^k$ is the quadrilateral region $X_i M_k O_kM_{k+1}$, respectively, and $M_7=M_1$. In addition, $\Omega_i = \bigcup\limits_{l=1}^{6} \tau_{l}$, which is the support of the Lagrangian basis function $\phi_i$.

\begin{figure}[h]
     \centering{
    \includegraphics[scale=0.35]{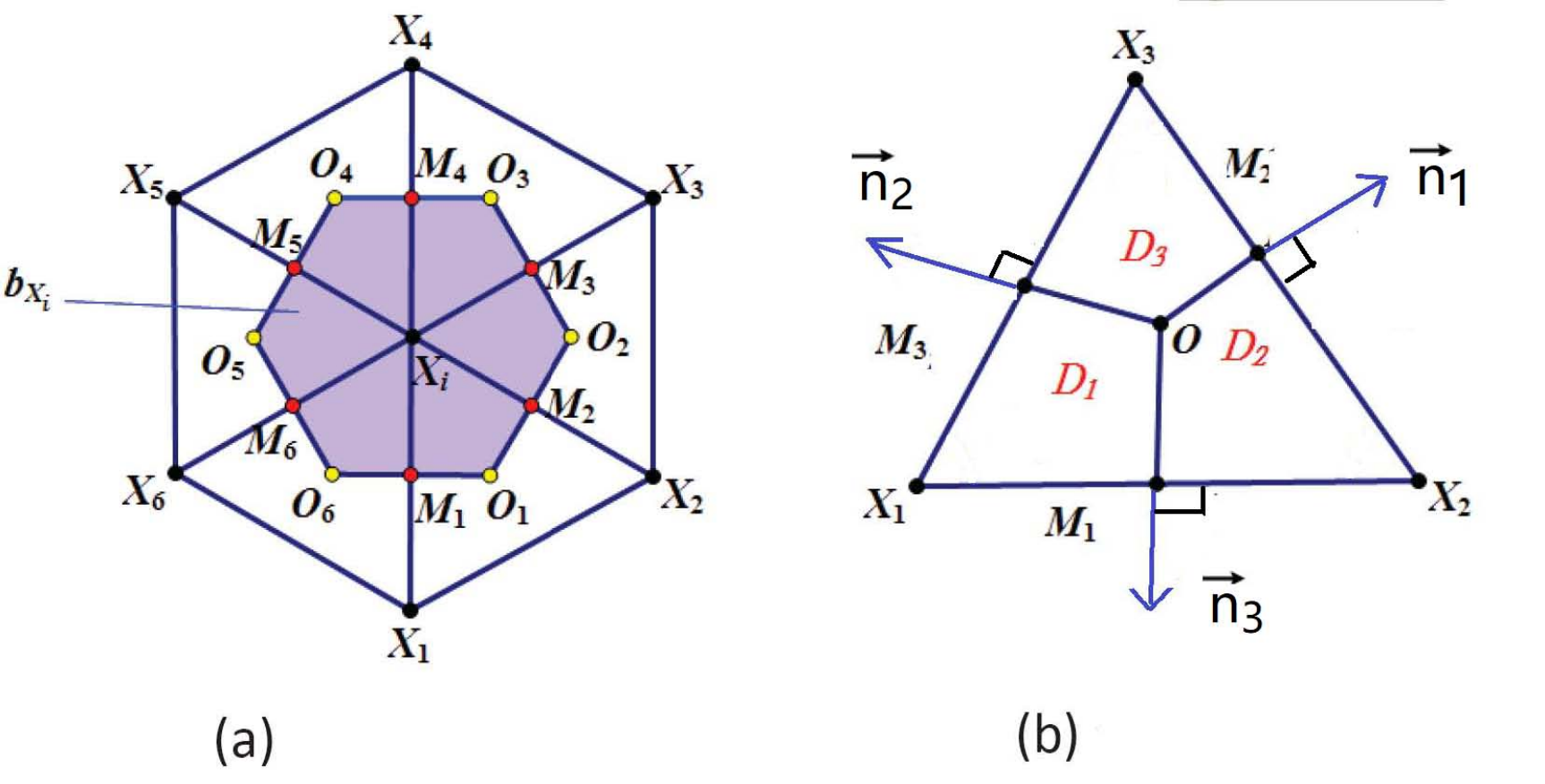}}
    \vskip -0.2cm
    \caption{(a) Dual element $b_{X_i}$. ~(b) Sub-control volume $D_i, i=1,2,3$ and outer normal vector $\vec{n}_i, i=1,2,3$ on three edges.}\label{kzt}
\end{figure}

To discretize the diffusion term, three fictitious elements will be introduced for $\tau_0 = \triangle X_1X_2X_i$ from its three edges. We will take one example for Edge $X_2X_i$ to
illustrate how to obtain its fictitious element. Another vertex $X_{1^{'}}$ of the fictitious $\triangle {X_{1^{'}}X_iX_2}$  (denoted as $\tau_1$) is forced to locate on the extend line of Line $OM_2$ (Shown as Fig. \ref{modp-1-new}). It is obvious that Points $X_1, O, M_2, X_{1^{'}}$ are collinear according to the choice of the barycenter point $O$ for the dual element. Furthermore, Point $O_1$ is chosen as the barycenter of $\triangle{X_{1^{'}}X_iX_2}$. Denote $\alpha_1= \frac{|X_{1^{'}}M_2|}{|X_1M_2|}$. As $\alpha_1$ is determined, Points $X_{1^{'}}, O_1$ are uniquely determined. Similar notations $X_{j^{'}}, \alpha_j, j=2$ are for another fictitious element $\tau_2=\triangle {X_{2^{'}}X_1X_i}$ (Seen in Fig. \ref{modp-1-new} (a)). 

\begin{figure}[h]
     \centering{
    \includegraphics[scale=0.25]{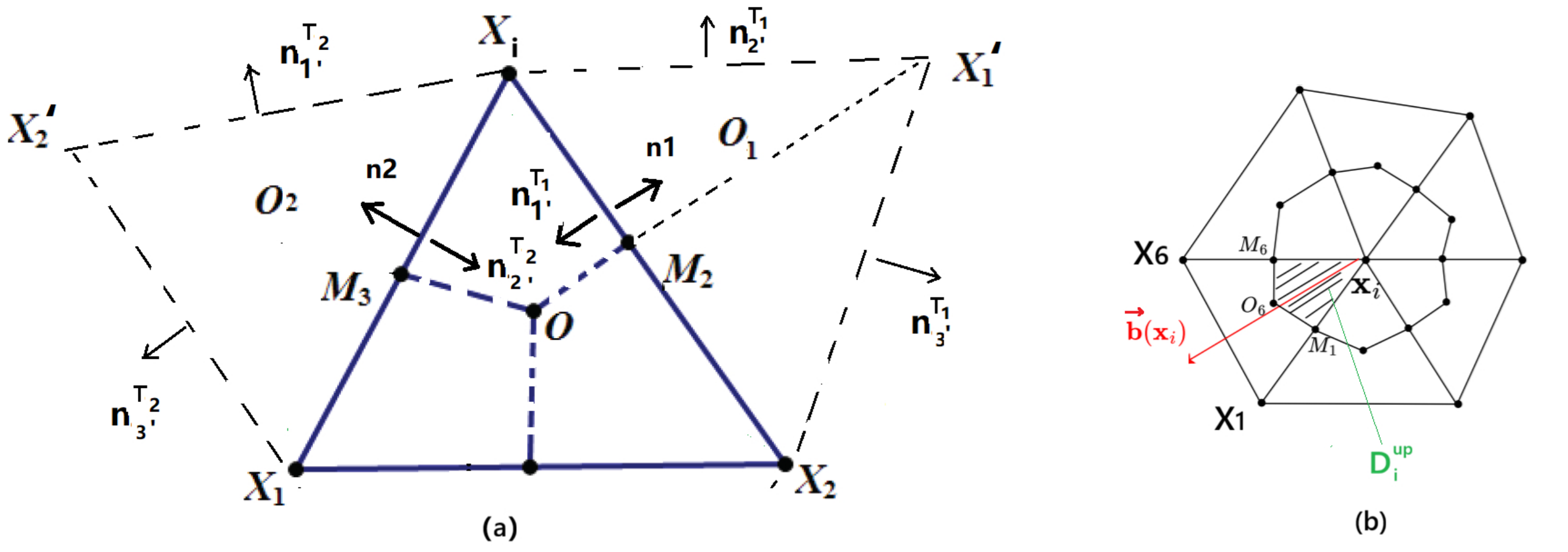}}
    \vskip -0.2cm
    \caption{(a) Element $E_k=\tau_0$ and its two fictitious elements $\tau_i,~i=1,2$ near to Node $X_i$, together with corresponding notations. (b) Upwind volume $D_i^{up}$.
     }\label{modp-1-new}
\end{figure}

To discretize the convection term, we introduce the definition of upwind volume.
For Node $\mathbf{x}_i$, there always exists one sub-control volume $D_i^{up} \in b_{X_i}=\cup_{k=1}^6 D_i^k$,  which is intersected by a ray passing through Node
$\mathbf{x}_i$ in the direction $-\mathbf{b}(\mathbf{x}_i)$. Such sub-region is called upwind volume denoted as $D_i^{up}$ (Shown as Fig. \ref{modp-1-new} (b)), which lies in one element denoted as $\tau_i^{up}$.


Let $U_h$ and $V_h$, respectively, be the trial and test function spaces
\begin{eqnarray}\label{U_h}
 U_h = \{ ~u_h \in C(\bar{\Omega}): ~ u_h|_{E_k} \in {P}_1(E_k), ~ u_h \big|_{\partial\Omega} = g, ~ \cup_{k=1}^M  E_k= \Omega_h \} \subset H^1(\Omega),
\end{eqnarray}
and
\begin{eqnarray*}
 V_h = \{ ~v_h \in L^2(\Omega): ~ v_h \in b_{X_i} = constant, b_{X_i} \in \Omega_h^*\},
\end{eqnarray*}
where
\begin{eqnarray}\label{psi-st}
  v_h(x) = \left\{
                    \begin{array}{ll}
                      1, & \hbox{ ~$\forall x \in b_{X_i} $, } \\
                      0, & \hbox{ ~$\forall x \notin b_{X_i} $.}
                    \end{array}
                  \right.
\end{eqnarray}

Integrating the equation (\ref{eq-PDE-1}) on Control volume $b_{X_i}\in \Omega_h^*$, applying the divergence theorem to the diffusion term, and approximating $u$ by $u_h \in U_h$,
we can obtain the approximate balance equation
 \begin{eqnarray}\label{app-blance}
 - \int_{\partial b_{X_i}}  \kappa  \nabla u_h \cdot  \vec{n} ds +   \int_{ b_{X_i}}  {\bf{b}} \cdot  \nabla u_h  d\mathbf{x} =  \int_{b_{X_i}} f d\mathbf{x},
\end{eqnarray}
where ~$\vec{n}$ is the unit outer normal vector on ~$\partial b_{X_i}$. 

\section{Discretization of the diffusion term}

In this section, we consider the first term $ - \int_{\partial b_{X_i}}  \kappa  \nabla u_h \cdot  \vec{n} ds$  in (\ref{app-blance}) corresponding to
the diffusion term. We firstly present the following result (\cite{Nie2}).

\begin{lemma}\label{grad-lamda}
For any element $\tau_0 = \triangle X_1X_2X_i$ (Shown as Fig. \ref{modp-1-new} (a)), it holds that
\begin{eqnarray*}
    \nabla \lambda_k^{\tau_0}=\frac{L_k}{2S_0}(-\vec{n}_k),~~k=1,2,i,
\end{eqnarray*}
where $ \lambda_k^{\tau_0}$ is the area coordinate corresponding to Node $X_k$, and
$\vec{n}_k$ is outer normal unit vector on the edge opposite to Node $X_k$, and $L_k$ is the length of the edge opposite to Node $X_k$, and $S_{0}$ is the area of Element $\tau_0$.
\end{lemma}

From Lemma \ref{grad-lamda}, 
\begin{eqnarray}\label{gradu0}
\nabla u_h|_{\tau_{0}}= \sum\limits_{k=1,2,i}  \nabla \lambda_k^{\tau_{0}} u_k^h =\sum\limits_{k=1,2,i}  \frac{L_k}{2S_0}(-\vec{n}_k) u_k^h.
\end{eqnarray}

Similarly, for Element $\tau_l$, 
\begin{eqnarray}\label{gradlambda}
    \nabla \lambda_{k^{'}}^{\tau_l}=\frac{L_{k^{'}}^{\tau_l}}{2S_1}(-\vec{n}_{k^{'}}^{\tau_l}),~~k=1,2,i,
\end{eqnarray}
where $\lambda_{k^{'}}^{\tau_l}$ is the area coordinate corresponding to Node $X_k^{'}$, and
$\vec{n}_{k^{'}}^{\tau_l}$ and $L_{k^{'}}^{\tau_l}$ are respectively, the outer normal unit vector and the length for the edge opposite to Node $X_{k^{'}}$ in $\tau_{l}$, and $S_l, l=1,2,3$ is area of $\tau_l$, respectively.

The first term in (\ref{app-blance}) is restricted to Element $E_k$, 
$$
\tilde{F}_i = -\int_{\overline{M_2O}}  \kappa \nabla  u  \cdot \vec{n} ds-\int_{\overline{OM}_3}  \kappa \nabla  u  \cdot \vec{n} ds.
$$

To obtain $\tilde{F}_i$, we will focus on the flux 
\begin{eqnarray}\label{flux}
F_i^1 = -\int_{\overline{O_1O}}  \kappa \nabla  u  \cdot \vec{n} ds.
\end{eqnarray}
From (\ref{gradu0}),
\begin{eqnarray}  \label{flux1}
F_{i1}^1 \approx -\int_{\overline{O_1O}} (\nabla u^h)|_{\tau_0} \cdot (\kappa^T \vec{n}_1) ds  = a_{1,1}u_1^h+a_{1,2}u_2^h+a_{1,i}u_i^h,
\end{eqnarray}
where
 $ a_{1,j}=\frac{L_j}{2S_0}\vec{n}_j \cdot (\kappa_0^T \vec{n}_{O_1O})|O_1O|,~~j=1,2,i$,  
and $\kappa_0^T=\kappa(O)^T$, and $|O_1O|$ and $\vec{n}_{O_1O}$ are the length and the unit outer normal vector, respectively. \\
From (\ref{gradlambda}),
\begin{eqnarray} \label{flux2}
F_{i2}^1\approx -\int_{\overline{O_1O}} (\nabla u^h)|_{\tau_1} \cdot (\kappa^T \vec{n}_1) dS  = b_{1,1}u_2^h+b_{1,i}u_i^h+b_{1,3}u_{1^{'}}^h
\end{eqnarray}
where
$b_{1,1}=\frac{L_{2^{'}}^{\tau_1}}{2S_1}\vec{n}_2^{'} \cdot (\kappa_1^T \vec{n}_{O_1O})|O_1O|, 
 b_{1,i}=\frac{L_{i^{'}}^{\tau_1}}{2S_1}\vec{n}_i^{'} \cdot (\kappa_1^T \vec{n}_{O_1O})|O_1O|,$ \\
$ b_{1,3}=\frac{L_{1^{'}}^{\tau_1}}{2S_1}\vec{n}_1^{'} \cdot (\kappa_1^T \vec{n}_{O_1O})|O_1O|$, ~~$\kappa_1^T=\kappa(O_1)^T$.

Taking some linear combination (\ref{flux1}) with (\ref{flux2}), we have
\begin{eqnarray}  \label{F1}
F_i^1 \approx \mu_1^1F_{i1}^1+\mu_1^2 F_{i2}^1= \alpha_{2,1} u_2^h - \alpha_{i,1} u_i^h  + \beta_1 u_1^h + \beta_{1^{'}} u_{1^{'}}^h,
\end{eqnarray}
where
$\alpha_{2,1}= \mu_1^1 g_{2,1}^{\tau_0} +\mu_1^2 g_{2,1}^{\tau_1},
\alpha_{i,1}= \mu_1^1 g_{i,1}^{\tau_0} + \mu_1^2 g_{i,1}^{\tau_1}, 
\beta_1= \mu_1^1 g_{1,1}^{\tau_0},~~~~ \beta_{1^{'}}= \mu_1^2 g_{1,1}^{\tau_1}$,
 $  g_{j,1}^{\tau_0}= \frac{L_j}{2S_0}\vec{n}_j \cdot (\kappa_0^T \vec{n}_{O_1O})|O_1O|,~
~~~g_{j,1}^{\tau_1} =\frac{L_{j^{'}}^{\tau_1}}{2S_1}\vec{n}_{j^{'}} \cdot (\kappa_1^T \vec{n}_{O_1O})|O_1O|, j=1,2,i$, \\
 $\mu_1^1+\mu_1^2 = 1$.\\

Let the term  $\beta_1 u_1^h + \beta_{1^{'}} u_{1^{'}}^h $ in (\ref{F1}) disappear to obtain
\begin{equation}\label{F1-exp}
F_i^1 = -\int_{\overline{O_1O}}  \kappa \nabla  u  \cdot \vec{n} ds \approx \alpha_{2,1} u_2^h - \alpha_{i,1} u_i^h,
\end{equation}
and
\begin{eqnarray*}\label{condition-2}
 \mu_1^1 g_{1,1}^{\tau_0} u_1^h + \mu_1^2 g_{1,1}^{\tau_1} u_{1^{'}}^h = 0.
\end{eqnarray*}
Hence,
\begin{eqnarray} \label{mu1-0}
  \mu_1^1 = \frac{ g_{1,1}^{\tau_1} u_{1^{'}}^h }{g_{1,1}^{\tau_1} u_{1^{'}}^h - g_{1,1}^{\tau_0} u_1^h},~~
  \mu_1^2 =  \frac{ -g_{1,1}^{\tau_0} u_1^h }{g_{1,1}^{\tau_1} u_{1^{'}}^h - g_{1,1}^{\tau_0} u_1^h}.
\end{eqnarray}

So far, we have obtained the two-point flux $F_i^1$ expressed as (\ref{F1-exp}) about Nodes $X_2$ and $X_i$.


Similarly, for Nodes $X_i$ and $X_1$, 
\begin{eqnarray}  \label{F2-exp}
F_i^2 = -\int_{\overline{O_2O}}  \kappa \nabla  u  \cdot \vec{n} ds \approx  \alpha_{i,2} u_i^h- \alpha_{1,1} u_1^h,  
\end{eqnarray}
where
$ 
\alpha_{i,2}= \mu_2^1 g_{i,2}^{\tau_0} +\mu_2^2 g_{i,2}^{\tau_2},~~~
\alpha_{1,1}= \mu_2^1 g_{1,2}^{\tau_0} + \mu_2^2 g_{1,2}^{\tau_2}$ ,$\kappa_2^T=\kappa(O_2)^T$,
\begin{eqnarray}\nonumber 
   g_{j,2}^{\tau_0}= \frac{L_j}{2S_0}\vec{n}_j \cdot (\kappa_0^T \vec{n}_{O_2O})|O_2O|, ~~g_{j,2}^{\tau_1} =\frac{L_{j^{'}}^{\tau_2}}{2S_2}\vec{n}_{j^{'}} \cdot (\kappa_2^T \vec{n}_{O_2O})|O_2O|, ~j=1,2,i,\\ \label{mu2-0}
  \mu_2^1 = \frac{ g_{2,2}^{\tau_2} u_{2^{'}}^h }{g_{2,2}^{\tau_2} u_{2^{'}}^h - g_{2,2}^{\tau_0} u_2^h}~~~
  \mu_2^2 = \frac{- g_{2,2}^{\tau_0} u_2^h }{g_{2,2}^{\tau_2} u_{2^{'}}^h - g_{2,2}^{\tau_0} u_2^h}.~~~
 \end{eqnarray}

The flux on Edge $\overline{M_2 OM}_{3}$~ can be written as
\begin{eqnarray} \label{tpflux}
 -\int_{ \overline{M_2OM_3}} \kappa \nabla u_h \cdot \vec{n}_k ds & = &  - \frac{1}{1+\alpha_2}\int_{ \overline{O_2O}}\kappa \nabla u_h \cdot \vec{n}_k ds  \\ \nonumber &~~~~& + \frac{1}{1+\alpha_3}\int_{ \overline{O_3O}}\kappa \nabla u_h \cdot \vec{n}_{k-1} ds.
\end{eqnarray}
Hence, from (\ref{F1-exp}), (\ref{F2-exp}),
\begin{eqnarray} \label{emat}
 -\int_{ \overline{M_2OM_3}} \kappa \nabla u_h \cdot \vec{n}_k ds 
= a_{i,i}^{E_k} u_i^h +a_{i,2}^{E_k}u_2^h + a_{i,1}^{E_k} u_1^h,
 \end{eqnarray}
where 
\begin{eqnarray*}
a_{i,i}^{E_k}= \frac{\alpha_{i,1}+\alpha_{i,2}}{1+\alpha_2},
a_{i,2}=- \frac{\alpha_{2,1}}{1+\alpha_2},
a_{i,1}=- \frac{\alpha_{1,1}}{1+\alpha_3}.
 \end{eqnarray*}

One can easily see that
\begin{eqnarray}  \label{fluxmo}
  -\int_{\overline{M_kO}} \nabla u \cdot (\kappa^T \vec{n}_k) ds = -\frac{1}{1+\alpha_k}\int_{\overline{O_kO}} \nabla u \cdot (\kappa^T \vec{n}_k) ds,~~k=2,3,
  \end{eqnarray}
 where $    \alpha_k=\frac{|X_k^{'}M_k|}{|X_kM_k|}$,
and its discretization is similar to (\ref{fluxmo}). 

For all sub-control volumes $D_i^{i_j} \in b_{X_i}$, from  (\ref{emat}) and (\ref{fluxmo}), one can easily infer that
\begin{eqnarray} \label{diff-operator}
      &  & - \int_{\partial b_{X_i}}  \kappa  \nabla u_h \cdot  \vec{n} ds 
       = - \sum\limits_{j=1}^{m_i}\int_{\partial D_i^{i_j}}  \kappa  \nabla u_h \cdot  \vec{n} ds  \approx \sum_{j=1}^{m_i} d_{i i_j} u_{i_j}^h,
\end{eqnarray}
where
$$
d_{i i_j}=\sum\limits_{D_i^{i_j} \in b_{X_i}} N_{E_k} a_{ij}^{E_k} N_{E_k}^T, 
$$
and $N_{E_k}$ is the assembling matrix consisting of zeros and ones.

We call \eqref{diff-operator} as the PFVE discrete operator for diffusion term.

\section{Discretization of the convection term}

Similar to the FE method, we can obtain the discrete operator for the convection (second) term in (\ref{app-blance}) by this way that all element convection matrices within $b_{X_i}$ are integrated into the total convection matrix $A_c$. It is well known that matrix $A_c$ is always difficult to satisfy the M-matrix property, especially for convection-dominated problems. 

\subsection{first-order discrete operator}

Consider the convective term in Equation (\ref{app-blance}), and notice $u_h\in U_h$, 
\begin{eqnarray} \label{convectterm}
\int_{ b_{X_i}}  {\bf{b}} \cdot  \nabla u_h  d \mathbf{x} \approx  \int_{ b_{X_i}}  {\bf{b}} \cdot  \nabla u_h|_{D_i^{up}}  d \mathbf{x} = \nabla u_h|_{D_i^{up}} \cdot \int_{ b_{X_i}}  {\bf{b}}   d\mathbf{x} .
\end{eqnarray}
where $\nabla u_h|_{D_i^{up}}$ is one constant vector. 

By some numerical integration formula, 
\begin{eqnarray} \label{convectterm2}
\int_{ b_{X_i}}  {\bf{b}}  dx \approx \sum\limits_{k=1}^m \omega_k {\bf b} (X_i^k)|b_{X_i}|,
\end{eqnarray}
where $|b_{X_i}|$ is the area of the dual element $b_{X_i}$, $X_i^k$ and $\omega_k $ are, respectively, the Gaussian integral point and weight.

Hence,
\begin{eqnarray} \label{convect-operator}
\int_{ b_{X_i}}  {\bf{b}} \cdot  \nabla u_h  dx \approx  
\sum\limits_{i_j=1,i,6}^3 c_{ii_j} u_{i_j}^h,
\end{eqnarray}
where 
$$
c_{ii_j}=|b_{X_i}| \sum\limits_{k=1}^m \omega_k {\bf b} (X_i^k)\cdot \nabla \lambda_{i_j},
$$
and $\lambda_{i_j}$ is the shape function of Node $X_{i_j}$ in $\triangle X_iX_6X_1$.

Assume that ${\bf b}$ is one constant vector, and ${\bf b}({X}_i) = \bf{b}_0$, 
\begin{eqnarray}\label{Tabata} 
c_{ij}=|b_{X_i}|  {\bf{b}}_0\cdot \nabla \lambda_{i_j}= \frac{|\Omega_i|}{3}{\bf{b}}_0 \cdot \nabla \lambda_{i_j},
\end{eqnarray}
as $b_{X_i}\in \Omega_h^*$ is the barycenter dual partition, which accords with that of the literature \cite{Tabata}, and it is first-order convergent.

\subsection{Second-order discrete operator}

In this subsection, we will construct one second-order monotonous nonlinear discrete operator. 
\subsubsection{Error estimates on the linear extension function}

We introduce one linear interpolation function in the region $D_i^{up}$ (For convenience, we write $D_i^{up}$ as $D^{up}$ in the following) and its extension function $\tilde{u}^{up}(\mathbf{x})$ in $b_{X_i}$ as follows.
\begin{equation}\label{eq-tau-up-t}
\tilde{u}^{up}(\mathbf{x}) = \sum_{l = 1}^3 P_l^{up}(\mathbf{x}) u(\mathbf{x}_ {i_l}),\forall \mathbf{x} \in b_{X_i},
\end{equation}
where $P_l^{up}(\mathbf{x})  = \lambda_l^{up}(\mathbf{x}) \in \mathcal{P}_1, l =1,2,3$, $\lambda_l^{up}$ is the linear Lagrange interpolation basis function in Element  $\tau_6$ (also denoted as $\tau_i^{up}$), $\mathbf{x}_{i_l} := (x_l, y_l)$, $\mathbf{x}_{i_1} = \mathbf{x}_i$.

For $\tilde{u}^{up}(\mathbf{x})$, we have the following  asymptotic expansion which  can be proved similarly to that in the literature (\cite{Nie3}).

\begin{lemma}\label{err-tau}
  Assume that $\tilde{u}^{up}(\mathbf{x})$, $u(\mathbf{x})\in C^2(b_{X_i})$ ( resp.,  $\nabla \tilde{u}^{up}(\mathbf{x})$, $\nabla    u(\mathbf{x})$ ) are, the extended function (resp., gradient functions ) defined as (\ref{eq-tau-up-t}), the solution of (\ref{eq-PDE-1}) restricted in $b_{X_i}$. Then

(1) For $\tilde{u}^{up}(\mathbf{x})$,
\begin{equation}\label{eq-uupI-u}
\tilde{u}^{up}(\mathbf{x})= u(\mathbf{x})+ g^{{up}}(\mathbf{x})+C_1h^3_i, \forall ~\mathbf{x} \in b_{X_i},
\end{equation}
where $C_1$ is a constant relative to the values $\frac{\partial^3 u}{\partial x^{\alpha} \partial y^{3-\alpha}}, \alpha = 0,1,2,3$, and $h_i$ is maximum length of all triangle edges in $b_{X_i}$,  and
\begin{eqnarray}\label{eq-g-tauup}
&g^{{up}}(\mathbf{x})=\frac{u_{xx}(\mathbf{x}_i)}{2}\sum\limits_{l = 1}^3 P_l^{{up}}(\mathbf{x})(x_l-x)^{2}
+\frac{u_{yy}(\mathbf{x}_i)}{2}\sum\limits_{l = 1}^3 P_l^{{up}}(\mathbf{x})(y_l-y)^2 \nonumber\\
&+ u_{xy}(\mathbf{x}_i)\sum\limits_{l = 1}^3 P_l^{{up}}(\mathbf{x})(x_l-x)(y_l-y) = O(h_i^2).
\end{eqnarray}

(2) For $\nabla \tilde{u}^{up}(\mathbf{x})$,
\begin{equation}\label{eq-graduI-gradu}
      { \nabla \tilde{u}^{up}(\mathbf{x})= \nabla u(\mathbf{x}) + \nabla \tilde{g}^{{up}}(\mathbf{x})+Ch^2_i, \forall \mathbf{x} \in b_{X_i}},
    \end{equation}
    where 
    \begin{eqnarray}\nonumber
      \nabla \tilde{g}^{{up}}(\mathbf{x})
    & =&   \frac{u_{xx}(\mathbf{x}_{i})}{2}\sum_{l=1}^3\nabla P_l^{{up}}(x_l-x)^{2}
  +\frac{u_{yy}(\mathbf{x}_{i})}{2}\sum_{l=1}^3\nabla P_l^{{up}}(y_l-y)^{2}
\\ \label{eq-u-uI-g-1}
  & &+u_{xy}(\mathbf{x}_i)\sum_{l=1}^3\nabla P_l^{{up}}(x_l-x)(y_l-y) = O(h_i).
    \end{eqnarray}
\end{lemma}

 From Lemma \ref{err-tau},
      $$
      |\nabla \tilde{u}^{up}(\mathbf{x}) - \nabla u(\mathbf{x})| = O(h_i), \forall \mathbf{x} \in b_{X_i}.
      $$
According to the fact that $\nabla \tilde{u}^{up}(\mathbf{x})$ defined by (\ref{eq-tau-up-t}) is one constant vector,
  $$
  \nabla {u}^{up} =  \nabla \tilde{u}^{up}|_{D_i^{up}} = \nabla \tilde{u}^{up}(\mathbf{x}), \forall \mathbf{x} \in b_{X_i}.
 $$

Therefore, we have the following result.
\begin{corollary}\label{err-grad-up}
Assume that 
 $u(\mathbf{x})\in C^2(b_{X_i})$ and
$  \nabla u^{up}(\mathbf{x}) = \sum\limits_{l = 1}^3 \nabla \lambda_l^{up} u(\mathbf{x}_ {i_l})$. Then
  \begin{equation}\label{eq-FE-up-L-2-p1}
    |\nabla u^{up}-\nabla u(\mathbf{x})| = O(h_i),\forall \mathbf{x}\in b_{X_i}.
  \end{equation}
  \end{corollary}

\subsubsection{ Error estimates on the discrete convection operator}

The convection operator 
\begin{equation}\label{eq-PDE-b}
  (L^cu,\psi_i) := (\mathbf{b} \cdot \nabla u, \psi_i) = \int_{b_{X_i}}\mathbf{b} \cdot\nabla u  d\mathbf{x},~~\psi_i \in V_h.
\end{equation}

The FVE discrete convection operator
\begin{equation}\label{eq-FE-up-L}
  (L_h^cu_h,\psi_i) =  \sum_{l=1}^3c_{ii_l}u_{i_l},
\end{equation}
where
\begin{equation}\label{eq-FE-up-c}
  c_{ii_l}  = |b_{X_i}| \sum\limits_{k=1}^m \omega_k {\bf b} (X_i^k)\cdot \nabla \lambda_{i_l},
\end{equation}
and $i_l, l = 1, 2, 3$ are the global indices of three nodes of Element $\tau_i^{up}$, 
and $\lambda_{i_l}$ is the shape function of Node $X_{i_l}\in\tau_i^{up}$, $X_i^k \in D_i^{up}$ is the Gaussian quadrature point.

In the following, we estimate the difference $ (L_h^cu(\mathbf{x}_{i}),\psi_i)-(L^cu,\psi_i)$.

\begin{proposition}\label{discret-conv}{ Assume that  $u(\mathbf{x})\in C^2(b_{X_i})$ and $(L^cu,\psi_i), (L_h^cu(\mathbf{x}_{i}), \psi_i)$ are, respectively, defined by (\ref{eq-PDE-b}), (\ref{eq-FE-up-L}). Then
\begin{eqnarray}\label{eq-FE-up-L-2}
 |(L_h^cu(\mathbf{x}_{i}),\psi_i)-(L^cu,\psi_i)| =  O(h_i^2), ~\mathbf{x}\in b_{X_i}.
\end{eqnarray}
Especially for a constant vector $\mathbf{b}$,
\begin{eqnarray}\label{eq-FE-up-L-5}
 |(L_h^cu(\mathbf{x}_{i}),\psi_i)-(L^cu,\psi_i)| =  O(h_i^3), ~\mathbf{x}\in b_{X_i}.
\end{eqnarray}}
\end{proposition}

 \begin{proof}
Firstly, we prove
  \begin{equation}\label{eq-FE-up-L-2-p2}
   |(L^cu,\psi_i)-(L^cu^{up},\psi_i)| = O(h_i^3),
  \end{equation}
where $u^{up}$ is the interpolation function in $D_i^{up}$.

As a fact, from (\ref{eq-PDE-b}),
 \begin{align*}
   (L^cu,\psi_i) - (L^c u^{up},\psi_i)
    &=  \int_{\Omega_i}\mathbf{b} \cdot\nabla u \psi_i d\mathbf{x}
     -\int_{\Omega_i}\mathbf{b} \cdot\nabla u^{up} \psi_i d\mathbf{x}\\
    &= \int_{\Omega_i}\mathbf{b}\cdot(\nabla u-\nabla u^{up}) \psi_i d\mathbf{x}
\\
    &= \int_{b_{X_i}}\mathbf{b}\cdot(\nabla u-\nabla u^{up}) d\mathbf{x}.
   \end{align*}
By Corollary \ref{err-grad-up}, 
   $$
   \mathbf{b}\cdot(\nabla u-\nabla u_{\tau^{up}})  = O(h_i).
   $$
Hence, (\ref{eq-FE-up-L-2-p2}) holds true from the fact $|b_{X_i}| = O(h_i^2)$.

 Then, we prove
  \begin{equation}\label{eq-FE-up-L-2-p3}
   |(L^c_hu(\mathbf{x}_{i}),\psi_i)-(L^cu^{up},\psi_i)| = O(h_i^2).
  \end{equation}

In fact, from (\ref{eq-PDE-b}) and (\ref{eq-FE-up-L}),
  \begin{align}
    &(L^c_hu(\mathbf{x}_{i}),\psi_i)-(L^cu^{up},\psi_i) \nonumber \\
   =&\sum_{l=1}^3c_{ii_l}^{up}u(\mathbf{x}_{i_l}) - \int_{\Omega_i}\mathbf{b}\cdot\nabla u^{up} \psi_i d\mathbf{x} \nonumber \\
   =&\sum_{l=1}^3  (|b_{X_i}|\sum_{k=1}^m \omega_k \mathbf{b}(\mathbf{x}_i^k) \cdot \nabla \lambda^{up}_{i_l}) u(\mathbf{x}_{i_l})  - \sum_{l = 1}^3  \int_{b_{X_i}}\mathbf{b}  dx \cdot\nabla \lambda_{i_l}^{up}  u(\mathbf{x}_ {i_l}) \nonumber \\
  =&\sum_{l=1}^3   (|b_{X_i}| \sum_{k=1}^m \omega_k \mathbf{b}(\mathbf{x}_i^k)-  \int_{b_{X_i}}\mathbf{b}  dx) \cdot\nabla \lambda_{i_l}^{up} u(\mathbf{x}_ {i_l})  . \label{asd-1}
  \end{align}

If $\mathbf{b}$ is one constant vector. Then
  \begin{equation}\label{aa1}
  (L^c_hu(\mathbf{x}_{i}),\phi_i)-(L^cu^{up},\phi_i) = 0.
  \end{equation}
Otherwise, by some proper composite Gaussian quadrature,
  $$
  \int_{b_{X_i}}\mathbf{b}  dx-|b_{X_i}| \sum_{k=1}^m \omega_k \mathbf{b}(\mathbf{x}_i^k)=O(h_i^{2+l}), l\ge 1,
  $$
together with $\nabla \lambda^{\tau^{up}}_{l} = O(h_i^{-1})$,
(\ref{eq-FE-up-L-2-p3}) holds true from (\ref{asd-1}).

Therefore, from the triangle inequality, (\ref{eq-FE-up-L-2}) is true from (\ref{eq-FE-up-L-2-p2}) and (\ref{eq-FE-up-L-2-p3}).
(\ref{eq-FE-up-L-5}) is true from (\ref{eq-FE-up-L-2-p2}) and (\ref{aa1}).

This completes the proof of this proposition.
\end{proof}

\subsubsection{Second-order discrete operator}
To construct the second-order discrete operator, some nonlinear coefficients $\beta_{i,l} = \beta_{i,l}(u_{i_1},u_{i_2},\cdots,u_{i_{m_i}}),l =1,2,3$, are introduced for (\ref{eq-FE-up-L}), we have
    \begin{eqnarray*}
     (\tilde{L}_h^cu_i,\phi_i) = \sum_{l=1}^3 (1+\beta_{i,l} h_i^2) c_{ii_l} u_{i_l},
    \end{eqnarray*}
where $i_l, l = 1, 2, 3$ are the global indices of three nodes of Element $\tau_i^{up}$ and
\begin{eqnarray}\label{cij1}
   c_{ii_l}  = |b_{X_i}| \sum\limits_{k=1}^m \omega_k {\bf b} (X_i^k)\cdot \nabla \lambda_{i_l}.
\end{eqnarray}

Assume that the perturbed interpolation function $\bar{u}^{up}(\mathbf{x})\in \mathcal{P}_1$,
  \begin{eqnarray}\label{utar}
    \bar{u}^{up}(\mathbf{x}) = \sum_{l = 1}^3 (1+\beta_{i,l}h^2_i )P_{i}^{up}(\mathbf{x}) u(\mathbf{x}_ {i_l}).
  \end{eqnarray}
From Lemma 2,
  \begin{eqnarray*}
      \bar{u}^{up}(\mathbf{x})- u(\mathbf{x}) = g^{up}(\mathbf{x}) + \sum_{l = 1}^3 (\beta_{i,l}h^2_i )P_{i}^{up}(\mathbf{x}) u(\mathbf{x}_ {i_l})+O(h^3_i), \forall \mathbf{x} \in b_{X_i},
  \end{eqnarray*}
  where $g^{up}(\mathbf{x})$ is defined by (\ref{eq-g-tauup}).

  The gradient error expansion
  \begin{eqnarray}\label{eq-u-u-up-I-beta}
      \nabla \bar{u}^{up}(\mathbf{x})- \nabla u(\mathbf{x}) = \nabla \tilde{g}^{up}(\mathbf{x})+ \sum_{l = 1}^3 (\beta_{i,l}h^2_i )\nabla P_{i}^{up} u(\mathbf{x}_ {i_l})+O(h^2_i), \forall \mathbf{x} \in b_{X_i},~~
  \end{eqnarray}
  where  $ \nabla \tilde{g}^{up}(\mathbf{x})$ is defined by (\ref{eq-u-uI-g-1}).

   \vskip 0.5cm
   \begin{theorem} Assume that $u\in C^2(b_{X_i}), 1 \leq i \leq N$. Then
    \begin{equation}\label{eq-no-up-L-r}
   ({L}^c\bar{u}^{up},\psi_i) -(L^cu,\psi_i) = r(\beta_{i,l})+O(h^4_i),
  \end{equation}
where $\bar{u}^{up} \in U_h$ is defined by (\ref{utar}),
   \begin{equation}\label{eq-no-up-L-r-1}
  r(\beta_{i,l})  = \int_{b_{X_i}}\mathbf{b}\cdot { \nabla \tilde{g}^{up} (\mathbf{x})}  d\mathbf{x} +
   h_i^2 \sum_{l = 1}^3 \beta_{i,l}  (\nabla \lambda_l^{up} \cdot \int_{b_{X_i}}\mathbf{b} d\mathbf{x}) {u(\mathbf{x}_ {i_l})}.
  \end{equation}
 \end{theorem}

\begin{proof}
  From (\ref{eq-PDE-b}) and (\ref{eq-u-u-up-I-beta}),
  \begin{align*}
   &({L}^c\bar{u}^{up},\psi_i)-(L^cu,\psi_i)\\
    = & \int_{\Omega_i}\mathbf{b} \cdot \nabla \bar{u}^{up} \psi_i d\mathbf{x} - \int_{\Omega_i}\mathbf{b} \cdot\nabla u  \psi_i d\mathbf{x} \\
    = &\int_{\Omega_i}\mathbf{b}\cdot(\nabla \bar{u}^{up}-\nabla u) \phi_i d\mathbf{x}\\
    = &\int_{\Omega_i}\mathbf{b}\cdot(\nabla \tilde{g}^{up}(\mathbf{x}) + \sum_{l = 1}^3 (\beta_{i,l}h^2_i )\nabla \lambda_l^{up} u(\mathbf{x}_ {i_l})+O(h^2_i)) \psi_i d\mathbf{x}\\
    = &\int_{b_{X_i}}\mathbf{b}\cdot(\nabla \tilde{g}^{up}(\mathbf{x}) + \sum_{l = 1}^3 (\beta_{i,l}h^2_i )\nabla \lambda_l^{up} u(\mathbf{x}_ {i_l})+O(h^2_i)) d\mathbf{x}\\
   = & r(\beta_{i,l})+O(h^4_i),
   \end{align*}
   where $r(\beta_{i,l})$ is defined by (\ref{eq-no-up-L-r-1}).

Noticing that $|b_{X_i}| = O(h^2_i)$, and utilizing (\ref{eq-u-uI-g-1}), one can obtain
   $$
   r(\beta_{i,l}) = O(h_i^3).
   $$
   Hence, (\ref{eq-no-up-L-r}) holds true.

This completes the proof of this theorem.
\end{proof}

The current goal is
  \begin{equation}\label{conclude-1}
   ({L}^c\bar{u}_{\tau^{up}},\phi_i) -(L^cu,\phi_i) = O(h^4_i).
  \end{equation}

According to (\ref{eq-no-up-L-r}) and (\ref{conclude-1}), we can obtain $\beta_{i,l}$ from
  \begin{eqnarray*} \label{eq-beta-6}
   \int_{b_{X_i}}\mathbf{b}\cdot {\nabla \tilde{g}^{up} (\mathbf{x})} d\mathbf{x}+
   h_i^2 \sum_{l = 1}^3 \beta_{i,l}  (\nabla \lambda_l^{up} \cdot \int_{b_{X_i}}\mathbf{b} d\mathbf{x}) {u(\mathbf{x}_ {i_l})} = 0.
  \end{eqnarray*}
If a special case: $\beta_{i,2} = \beta_{i,3} = 0$, is considered, then
    \begin{eqnarray*}
   \int_{b_{X_i}}\mathbf{b}\cdot { \nabla \tilde{g}^{up} (\mathbf{x})}  d\mathbf{x} +
   h^2_i \beta_{i,1}  \nabla \lambda_1^{up} \cdot \int_{b_{X_i}}\mathbf{b} d\mathbf{x} {u(\mathbf{x}_{i_1})} = 0.
  \end{eqnarray*}
   Hence,
  \begin{eqnarray}\label{eq-beta-2}
    \beta_i := \beta_{i,1} = -\frac{\int_{b_{X_i}}\mathbf{b}\cdot { \nabla \tilde{g}^{up} (\mathbf{x})} d\mathbf{x}}{h^2_i (\nabla \lambda_1^{up} \cdot \int_{b_{X_i}}\mathbf{b} d\mathbf{x}){u(\mathbf{x}_ {i_1})}}.
  \end{eqnarray}

 To obtain $\beta_i$, we have to prepare the values of ${u(\mathbf{x}_ {i_1})}$ and $\nabla \tilde{g}^{up}(\mathbf{x})$. 
Interpolate or perform the least-square fitting to approximate $u(\mathbf{x}), \mathbf{x}=(x,y)\in b_{X_i}$ on the data $(\mathbf{x}_{il}, u_{il})$ with
$$
 u(\mathbf{x}) \approx  P_2^u(x,y) = \sum_{j+k =0}^2 b_{jk}x^jy^k,
$$
where $b_{jk}$ are the coefficients.

Thus, we obtained one second-order nonlinear discrete operator for the convective term,  which satisfies (\ref{conclude-1}),
\begin{equation}\label{eq-non-up}
 (\tilde{L}_h^cu_i,\phi_i) = (1+\beta_i(u_{i_1},u_{i_2},\cdots,u_{i_{m_i}})h^2_i) c_{ii_1} u_{i_1} +  \sum_{l=2}^3 c_{ii_l} u_{i_l}, 
\end{equation}
where $c_{ii_l}$, $\beta_i$ is defined by (\ref{cij1}),(\ref{eq-beta-2}), respectively. 

\subsection{MFVE scheme}

From \eqref{app-blance}, (\ref{diff-operator}) and (\ref{convect-operator}), we can obtain one first-order monotonous finite volume element scheme (called as MFVE-1 Scheme) in the following
\begin{eqnarray} \label{firstshceme}
       \sum_{j=1}^{m_i} d_{i i_j} u_{i_j}^h + \sum\limits_{i_l=1,i,6}^3 c_{ii_l} u_{i_j}^h=f_i,
\end{eqnarray}
where  $d_{ii_j}$ and $c_{i i_j}$ are defined in  (\ref{diff-operator}) and (\ref{convect-operator}), respectively, and
\begin{eqnarray}\label{fvector}
f_i = \int_{b_{X_i}} f  d\mathbf{x}.
\end{eqnarray}

From \eqref{app-blance}, \eqref{diff-operator} and \eqref{eq-non-up}, one second-order monotonous finite volume element scheme (called as MFVE-2 Scheme) in the following
\begin{equation}\label{eq-non-up-1}
 \sum_{j=1}^{m_i} d_{i i_j} u_{i_j}+ c_{ii_1}^{up}u_{i_1}^h + \sum_{i_l=1,6} c_{ii_l}u_{i_l}=f_i,
\end{equation}
 where
$$c_{ii_1}^{up}  := (1+\beta_i(u_{i_1},u_{i_2},\cdots,u_{i_{m_i}})h^2_i)c_{ii_1}^{up}.
$$

One can see that the sign properties of above coefficients depend on $\beta_i$. 

If $\beta_i \ge 0$, we have
\begin{equation}\label{beta>0}
  c_{ii_1}^{up}  > 0,~~c_{ii_l} \le 0,l =2,3,~~c_{ii_1}^{up} + \sum_{l = 2}^3 c_{ii_l} \ge 0.
\end{equation}

If $\beta_i < 0$, then one parameter is introduced  
  $$
\alpha > \max\{-c_{ii_1}^{up},0\},
 $$ 
such that 
\begin{equation*}
\sum_{j=1}^{m_i} d_{i i_j} u_{i_j} +  (c_{ii_1}^{up}+\alpha) u_{i_1} +  \sum_{l=2}^3 c_{ii_j} u_{i_j}  = f_i + \alpha u_{i_1}.
\end{equation*}
This yields
\begin{equation}\label{beta<0}
  c_{ii_1}^{up} + \alpha > 0,~~c_{ii_l} \le 0,l =2,3,~~c_{ii_1}^{up} + \alpha + \sum_{l = 2}^3 c_{ii_1} \ge 0.
\end{equation}

Let the global discrete system of scheme \eqref{eq-non-up-1} on $\Omega_h$ be   denoted by
\begin{eqnarray}\label{AdAc-Tabate-SEP}
  A^2 u_2^h = f,~~A^2 = A_d +A_c^2,
\end{eqnarray}
where $A^2 = (a_{ij}^2)_{N \times N}$ with $a_{ij}^2 = d_{ij} + c_{ij}^2$, $A^d = (d_{ij})_{N \times N}$ and $A^c = (c_{ij}^2)_{N \times N}$ are the coefficient matrices corresponding to the diffusion term and convection term corresponding to  MFVE-2 scheme.

\begin{theorem}\label{convent-2}
Assume that $\Omega_h$ is a quasi-uniform mesh with size $h$, and the matrix $A_d$ defined in \eqref{AdAc-Tabate-SEP} is an $M$-matrix, and that the solution $u \in C^2(\Omega) \cap H^1(\Omega)$ of problem \eqref{eq-PDE-1}. Then, the approximate solution  $u_h^{2}$ of  MFVE-2 scheme is monotonous and satisfies 
\begin{equation}\label{eq-FE-up-L-2-p1}
||u-u_h^{2}||_{L_2} = O(h^2).
\end{equation}
\end{theorem}

\begin{corollary}\label{err-grad-up}
Under the assumptions of Theroem \ref{convent-2}, the approximate solution  $u_h^{1}$ of  MFVE-1 scheme is monotonous and satisfies 
\begin{equation}\label{err2}
||u-u_h^{1}||_{L_2} = O(h).
\end{equation}
  \end{corollary}

\section{Numerical experiments}

In this subsection, we will focus on the convergence, adaptivity and monotonous properties on two typical meshes: Random meshes and Kershaw meshes \cite{z-2} (Shown as Fig. \ref{meshes}), respectively. 

In numerical experiments, the diffusion term is discretized by PFVE method \eqref{diff-operator}, while the convection term is discretized by FVE methods, i.e., the convection-dominated diffusion problems are solved by MFVE-1, MFVE-2 schemes, respectively. The corresponding approximate solutions are denoted as $u_h^{1}$ and $u_h^{2}$, respectively, and $h$  is the step size of $\Omega_h$. The Picard iteration will stop if $\frac{\|\mathbf{r}_k\|_0}{\|\mathbf{r}_0\|_0}< 10^{-5}$, where the residuals $ \mathbf{r}_k= \mathbf{f}-A (U^k) U^k$, $\mathbf{r}_0= \mathbf{f}-A(U^0) U^0$, and $\| \cdot \|_0$ is the discrete $l_2$ norm:
$\|\mathbf{v}\|_0 = (\frac{1}{N}\sum\limits_{i=1}^N v_i^2)^\frac{1}{2}$.
The linear solver is Bicgstable, and the linear iteration will stop if the norm of the relative residual for the corresponding linear system is less than $10^{-10}$. 

\begin{figure}[H]
     \centering{
    \includegraphics[scale=0.65]{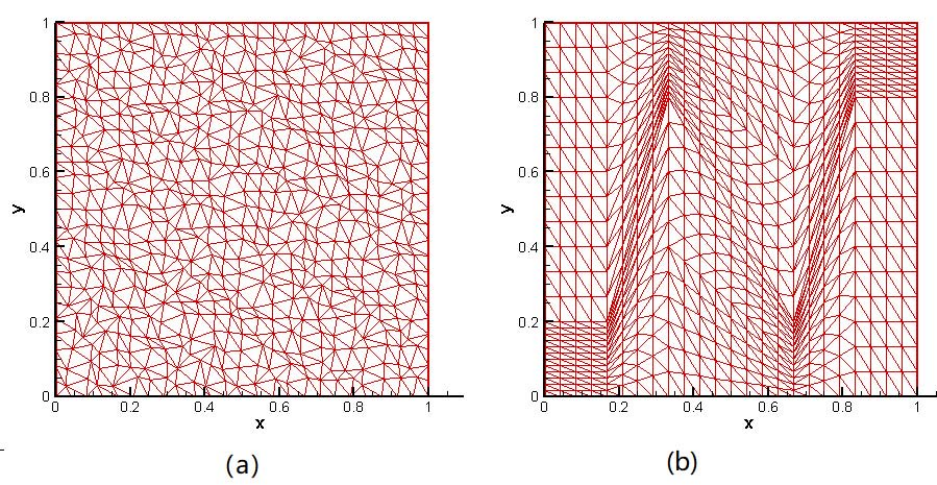}}
    \vskip -0.2cm
    \caption{(a)  Random meshes.
    ~(b) Kershaw meshes. }\label{meshes}
\end{figure}

\subsection{The convergence of new schemes}

In the following, we will focus on the stability and convergence of MFVE-1 and MFVE-2 schemes for convection-dominated diffusion problems.

\begin{example} \label{eg-4}
In model problem (\ref{eq-PDE-1}), we choose $\Omega = [0,1]^2$, the exact solution
$$
u(x,y) = e^{-\frac{(x-1)^2 + (y-1)^2}{2 \eta^2}},~~\eta = 0.1,
$$
and $g_1= u|_{\partial \Omega}$, $\mathbf{b} = (2,3)^{\top}$, and  $\varepsilon_k =10^{-k}I_{2 \times 2}$, $0 \leq k \leq 3$.
\end{example}

MFVE-1 and MFVE-2 schemes are employed, where approximate solutions (resp. errors) are, respectively, denoted as $u_h^1$, $u_h^2$ (resp. $e_h^1= u-u_h^1$,  $e_h^2= u-u_h^2$).
Numerical results are shown in Tables from \ref{tab-1} to \ref{tab-8}
for two meshes and different $\varepsilon_k$, $0 \leq k \leq 3$, where $\|\cdot\|_k, k=0,1$ is, respectively, the $L_2$, $H_1$ norm of the corresponding approximate errors. No further details for similar denotations will be described subsequently.‌
Tables 1, 3, 5, 7 are for Random meshes,  and Tables 2, 4, 6, 8 are for Kershaw meshes.

\begin{table}[H]
{\small

\caption{ The errors of $u_h^{1},u_h^{2}$ in norms $L_2$ and $H_1$ for 
$\varepsilon_0$ on random meshes.}\label{tab-1}
	\centering
\begin{tabular}{ccccccc cc }
\hline
nx &  $\|e_h^{1}\|_0$    & $\gamma_0$ &   $\|e_h^{1}\|_1$    & $\gamma_1$ & $\|e_h^{2}\|_0$    & $\gamma_0$ & $\|e_h^{2}\|_1$    & $\gamma_1$\\ \hline
24    & {7.18e-4} &      & {2.66e-1} &      & {9.26e-4} &       &2.65e-1  &\\
48   & { 2.61e-4} & 2.75 & {1.39e-1} & 1.91 & {2.38e-4} & 3.89  &1.39e-1  & 1.91\\
96   & { 1.32e-4} & 1.98 & {6.97e-2} & 1.99 & {5.62e-5} & 4.23  &6.97e-2  & 1.99\\
192   & {7.10e-5} & 1.86 & {3.50e-2} & 1.99 & {1.41e-5} & 3.99  &3.51e-2  & 1.99 \\\hline
\end{tabular}}
\end{table}

\begin{table}[H]
{\small
\caption{ The errors of $u_h^{1},u_h^{2}$ in norms $L_2$ and $H_1$ for 
$\varepsilon_0$ on Kershaw meshes.}\label{tab-2}
	\centering
\begin{tabular}{ccccccc cc }
\hline
nx &  $\|e_h^{1}\|_0$    & $\gamma_0$ &   $\|e_h^{1}\|_1$    & $\gamma_1$ & $\|e_h^{2}\|_0$    & $\gamma_0$ & $\|e_h^{2}\|_1$    & $\gamma_1$\\ \hline
24    & {2.99e-3} &      & {2.76e-1} &      & {2.30e-3} &       &2.76e-1  &\\
48   & { 1.24e-3} & 2.41 & {1.44e-1} & 1.92 & {8.23e-4} & 2.79  &1.44e-1  & 1.92\\
96   & { 4.91e-4} & 2.52 & {7.34e-2} & 1.96 & {2.52e-4} & 3.27  &7.34e-2  &1.96\\
192   & {2.05e-4} & 2.39 & {3.70e-2} & 1.98 & {6.94e-5} & 3.63  &3.69e-2  &1.98\\\hline
\end{tabular}}
\end{table}


\begin{table}[H]
{\small
\caption{ The errors of  $u_h^{1},u_h^{2}$ in norms $L_2$ and $H_1$ for 
$\varepsilon_1$  on random meshes.}\label{tab-3}
	\centering
\begin{tabular}{ccccccc cc }
\hline
nx &  $\|e_h^{1}\|_0$    & $\gamma_0$ &   $\|e_h^{1}\|_1$    & $\gamma_1$ & $\|e_h^{2}\|_0$    & $\gamma_0$ & $\|e_h^{2}\|_1$    & $\gamma_1$\\ \hline
24    & {3.32e-3} &      & {2.72e-1} &      & {3.29e-4} &       &2.66e-1  &\\
48   & { 1.83e-3} & 1.81 & {1.42e-1} & 1.92 & {1.08e-4} & 3.05  &1.39e-1  & 1.91\\
96   & { 9.66e-4} & 1.89 & {7.17e-2} & 1.98 & {2.86e-5} & 3.78  &6.97e-2  & 1.99 \\
192   & {4.96e-4} & 1.95 & {3.62e-2} & 1.98 & {7.51e-6} & 3.81  &3.51e-2  & 1.99 \\\hline
\end{tabular}}
\end{table}

\begin{table}[H]
{\small
\caption{ The errors of  $u_h^{1},u_h^{2}$ in norms $L_2$ and $H_1$ for 
$\varepsilon_1$  on Kershaw meshes.}\label{tab-4}
	\centering
\begin{tabular}{ccccccc cc }
\hline
nx &  $\|e_h^{1}\|_0$    & $\gamma_0$ &   $\|e_h^{1}\|_1$    & $\gamma_1$ & $\|e_h^{2}\|_0$    & $\gamma_0$ & $\|e_h^{2}\|_1$    & $\gamma_1$\\ \hline
24    & {7.91e-3} &      & {2.88e-1} &      & {2.52e-3} &       &2.78e-1  &\\
48   & { 4.10e-3} & 1.93 & {1.51e-1} & 1.91 & {8.43e-4} & 2.99  &1.44e-1  & 1.93\\
96   & { 2.07e-3} & 1.98 & {7.77e-2} & 1.94 & {2.42e-4} & 3.48  &7.35e-2  & 1.96\\
192   & {1.03e-3} & 2.01 & {3.93e-2} & 1.98 & {6.43e-5} & 3.76  &3.69e-2  & 1.99\\\hline
\end{tabular}}
\end{table}


\begin{table}[H]
{\small
\caption{ The errors of $u_h^{1},u_h^{2}$ in norms $L_2$ and $H_1$ for 
$\varepsilon_2$  on random meshes.}\label{tab-5}
	\centering
\begin{tabular}{ccccccc cc }
\hline
nx &  $\|e_h^{1}\|_0$    & $\gamma_0$ &   $\|e_h^{1}\|_1$    & $\gamma_1$ & $\|e_h^{2}\|_0$    & $\gamma_0$ & $\|e_h^{2}\|_1$    & $\gamma_1$\\ \hline
24    & {6.59e-3} &      & {3.02e-1} &      & {2.27e-3} &       &2.83e-1  & \\
48   & { 3.56e-3} & 1.85 & {1.69e-1} & 1.79 & {6.33e-4} & 3.58  &1.43e-1  & 1.98\\
96   & { 1.88e-3} & 1.89 & {9.24e-2} & 1.83 & {1.76e-4} & 3.60  &7.03e-2  & 2.03\\
192   & {9.72e-4} & 1.93 & {4.96e-2} & 1.86 & {4.74e-5} & 3.71  &3.51e-2  & 2.00\\\hline
\end{tabular}}
\end{table}

\begin{table}[H]
{\small
\caption{ The errors of $u_h^{1},u_h^{2}$ in norms $L_2$ and $H_1$ for 
$\varepsilon_2$ on Kershaw meshes.}\label{tab-6}
	\centering
\begin{tabular}{ccccccc cc }
\hline
nx &  $\|e_h^{1}\|_0$    & $\gamma_0$ &   $\|e_h^{1}\|_1$    & $\gamma_1$ & $\|e_h^{2}\|_0$    & $\gamma_0$ & $\|e_h^{2}\|_1$    & $\gamma_1$\\ \hline
24    & {1.57e-2} &      & {3.85e-1} &      & {5.56e-3} &       &3.14e-1  &\\
48   & { 7.82e-3} & 2.00 & {2.08e-1} & 1.85 & {1.54e-3} & 3.61  &1.52e-1  & 2.06\\
96   & { 3.88e-3} & 2.01 & {1.09e-1} & 1.91 & {4.51e-4} & 3.41  &7.47e-2  & 2.03 \\
192   & {1.93e-3} & 2.01 & {5.63e-2} & 1.94 & {1.13e-4} & 3.99  &3.72e-2  & 2.01\\\hline
\end{tabular}}
\end{table}



\begin{table}[H]
{\small
\caption{ The errors of $u_h^{1},u_h^{2}$ in norms $L_2$ and $H_1$ for 
$\varepsilon_3$  on random meshes.}\label{tab-7}
	\centering
\begin{tabular}{ccccccc cc }
\hline
nx &  $\|e_h^{1}\|_0$    & $\gamma_0$ &   $\|e_h^{1}\|_1$    & $\gamma_1$ & $\|e_h^{2}\|_0$    & $\gamma_0$ & $\|e_h^{2}\|_1$    & $\gamma_1$\\ \hline
24    & {7.29e-3} &      & {3.17e-1} &      & {6.24e-3} &       &4.28e-1  &\\
48   & { 3.95e-3} & 1.85 & {1.91e-1} & 1.66 & {2.33e-3} & 2.68  &2.63e-1  & 1.63\\
96   & { 2.08e-3} & 1.90 & {1.18e-1} & 1.62 & {6.56e-4} & 3.55  &8.64e-2  & 3.04 \\
192   & {1.08e-3} & 1.93 & {6.67e-2} & 1.77 & {1.73e-4} & 3.79  &4.03e-2  & 2.14\\\hline
\end{tabular}}
\end{table}

\begin{table}[H]
{\small
\caption{ The errors of $u_h^{1},u_h^{2}$ in norms $L_2$ and $H_1$ for 
$\varepsilon_3$  on Kershaw meshes.}\label{tab-8}
	\centering
\begin{tabular}{ccccccc cc }
\hline
nx &  $\|e_h^{1}\|_0$    & $\gamma_0$ &   $\|e_h^{1}\|_1$    & $\gamma_1$ & $\|e_h^{2}\|_0$    & $\gamma_0$ & $\|e_h^{2}\|_1$    & $\gamma_1$\\ \hline
24    & {1.72e-2} &      & {4.44e-1} &      & {9.37e-3} &       &5.20e-1  &\\
48   & { 8.42e-3} & 2.04 & {2.59e-1} & 1.71 & {3.71e-3} & 2.53  &3.49e-1  & 1.49\\
96   & { 4.18e-3} & 2.01 & {1.53e-1} & 1.70 & {6.35e-4} & 5.84  &1.07e-1  & 3.26 \\
192   & {2.08e-3} & 2.01 & {8.98e-2} & 1.71 & {1.67e-4} & 3.80  &4.00e-2  & 2.67 \\\hline
\end{tabular}}
\end{table}


The results of columns from 2 to 5 in Tab. \ref{tab-1} to Tab.\ref{tab-8} show that MFVE-1 scheme is stable and first-order convergent in both $L_2$ and $H_1$ norms for convection-dominated cases whether Random meshes or Kershaw meshes,  and the norm $H_1$ of error is approaches first-order as the $\varepsilon$ becomes smaller, i.e., the convection term dominates greatly. 

The results of  columns from 6 to 9 in Tab. \ref{tab-1} to Tab.\ref{tab-8} show that MFVE-2 scheme is stable and, respectively, second-order (resp. first-order ) convergent in $L_2$ (resp. $H_1$) norm for convection-dominated cases whether Random meshes or Kershaw meshes,  i.e., MFVE-2 is approximately saturated convergent order as the $\varepsilon$ becomes smaller, i.e., the convection term dominates greatly. 

\subsection{The adapativity of new schemes}

In the following, we will focus on the adaptivity of MFVE-1 and MFVE-2 schemes for convection-dominated diffusion problems.

\begin{example}
In model problem {(\ref{eq-PDE-1})}, we choose $\Omega = [0,1]^2$, $\mathbf{b} = (-1,1)^{\top}$, and  the exact solution
$$
u(x,y) = cos(\frac{\pi x}{2})e^y,
$$
$g= u|_{\partial \Omega}$, and the anisotropy diffusive tensor
$$
\kappa =
\left(
\begin{array}{cc}
10 \epsilon_k  &  0 \\
     0          & 0.1 \epsilon_k
\end{array}
\right),  \epsilon_k = 10^{-k}, 0 \leq k \leq 4.
$$
\end{example}  

MFVE-1 and MFVE-2 schemes are employed, where approximate errors are, respectively, denoted as $e_h^1= u-u_h^1$ and $e_h^2= u-u_h^2$.
Numerical results are shown in Tables from \ref{tab-9} to \ref{tab-16}
for different $\varepsilon_k$, $0 \leq k \leq 3$, where $\|\cdot\|_k, k=0,1$ is, respectively, the $L_2$, $H_1$ norm of the corresponding approximate errors.
Tables 9, 11, 13, 15 are for Random meshes,  and Tables 10, 12, 14, 16 are for Kershaw meshes.

\begin{table}[H]
{\small
\caption{ The errors of $u_h^{1},u_h^{2}$ in norms $L_2$ and $H_1$ for 
$\varepsilon_0$  on random meshes.}\label{tab-9}
	\centering
\begin{tabular}{ccccccc cc }
\hline
nx &  $\|e_h^{1}\|_0$    & $\gamma_0$ &   $\|e_h^{1}\|_1$    & $\gamma_1$ & $\|e_h^{2}\|_0$    & $\gamma_0$ & $\|e_h^{2}\|_1$    & $\gamma_1$\\ \hline
24    & {6.16e-4} &      & {8.74e-2} &      & {4.56e-4} &       &8.74e-2  &\\
48   & { 2.29e-4} & 2.71 & {4.41e-2} & 1.98 & {1.36e-4} & 3.35  &4.42e-2  & 1.91\\
96   & { 9.32e-5} & 2.45 & {2.21e-2} & 1.99 & {4.01e-5} & 3.39  &2.21e-2  & 2.00\\
192   & {4.05e-5} & 2.29 & {1.11e-2} & 1.99 & {1.05e-5} & 3.82  &1.10e-2  & 2.01\\\hline
\end{tabular}}
\end{table}

\begin{table}[H]
{\small
\caption{ The errors of $u_h^{1},u_h^{2}$ in norms $L_2$ and $H_1$ for 
$\varepsilon_0$  on Kershaw meshes.}\label{tab-10}
	\centering
\begin{tabular}{ccccccc cc }
\hline
nx &  $\|e_h^{1}\|_0$    & $\gamma_0$ &   $\|e_h^{1}\|_1$    & $\gamma_1$ & $\|e_h^{2}\|_0$    & $\gamma_0$ & $\|e_h^{2}\|_1$    & $\gamma_1$\\ \hline
24    & {5.44e-3} &      & {1.47e-1} &      & {5.43e-4} &       &1.48e-1  & \\
48   & { 2.06e-3} & 2.68 & {7.27e-2} & 2.02 & {2.03e-4} & 2.36  &7.28e-2  & 2.01\\
96   & { 6.49e-4} & 3.17 & {3.58e-2} & 2.03 & {6.25e-5} & 3.25  &3.58e-2  & 2.03\\
192   & {3.16e-4} & 2.05 & {1.81e-2} & 1.98 & {1.63e-5} & 3.83  &1.75e-2  & 2.04\\\hline
\end{tabular}}
\end{table}


\begin{table}[H]
{\small
\caption{ The errors of $u_h^{1},u_h^{2}$ in norms $L_2$ and $H_1$ for 
$\varepsilon_1$  on random meshes.}\label{tab-11}
	\centering
\begin{tabular}{ccccccc cc }
\hline
nx &  $\|e_h^{1}\|_0$    & $\gamma_0$ &   $\|e_h^{1}\|_1$    & $\gamma_1$ & $\|e_h^{2}\|_0$    & $\gamma_0$ & $\|e_h^{2}\|_1$    & $\gamma_1$\\ \hline
24    & {2.86e-3} &      & {8.92e-2} &      & {4.29e-4} &       &8.75e-2  &\\
48   & { 1.39e-3} & 2.06 & {4.55e-2} & 1.96 & {1.17e-4} & 3.67  &4.41e-2  & 1.98\\
96   & { 6.99e-4} & 1.99 & {2.31e-2} & 1.97 & {3.38e-5} & 3.46  &2.21e-2  & 1.99 \\
192   & {3.49e-4} & 2.00 & {1.17e-2} & 1.97 & {9.03e-6} & 3.74  &1.11e-2  & 1.99 \\\hline
\end{tabular}}
\end{table}

\begin{table}[H]
{\small
\caption{ The errors of $u_h^{1},u_h^{2}$ in norms $L_2$ and $H_1$ for 
$\varepsilon_1$  on Kershaw meshes.}\label{tab-12}
	\centering
\begin{tabular}{ccccccc cc }
\hline
nx &  $\|e_h^{1}\|_0$    & $\gamma_0$ &   $\|e_h^{1}\|_1$    & $\gamma_1$ & $\|e_h^{2}\|_0$    & $\gamma_0$ & $\|e_h^{2}\|_1$    & $\gamma_1$\\ \hline
24    & {5.66e-3} &      & {1.45e-1} &      & {4.97e-3} &       &1.47e-1  &\\
48   & { 2.39e-3} & 2.36 & {7.23e-2} & 2.00 & {1.74e-3} & 2.85  &7.26e-2  &2.02\\
96   & { 9.65e-4} & 2.31 & {3.59e-2} & 2.01 & {5.10e-4} & 3.41  &3.57e-2  &2.03 \\
192   & {4.12e-4} & 2.34 & {1.79e-2} & 2.01 & {1.38e-4} & 3.70  &1.77e-2  &2.01\\\hline
\end{tabular}}
\end{table}


\begin{table}[H]
{\small
\caption{ The errors of $u_h^{1},u_h^{2}$ in norms $L_2$ and $H_1$ for 
$\varepsilon_2$  on random meshes.}\label{tab-13}
	\centering
\begin{tabular}{ccccccc cc }
\hline
nx &  $\|e_h^{1}\|_0$    & $\gamma_0$ &   $\|e_h^{1}\|_1$    & $\gamma_1$ & $\|e_h^{2}\|_0$    & $\gamma_0$ & $\|e_h^{2}\|_1$    & $\gamma_1$\\ \hline
24    & {1.17e-2} &      & {1.36e-1} &      & {5.01e-4} &       &8.86e-2  &\\
48   & { 5.97e-3} & 1.96 & {8.25e-2} & 1.65 & {1.85e-4} & 2.71  &4.41e-2  & 2.01\\
96   & { 3.07e-3} & 1.96 & {5.21e-2} & 1.59 & {4.78e-5} & 3.87  &2.20e-2  & 2.00 \\
192   & {1.55e-3} & 1.98 & {3.30e-2} & 1.58 & {1.29e-5} & 3.71  &1.10e-2  & 2.00\\\hline
\end{tabular}}
\end{table}

\begin{table}[H]
{\small
\caption{ The errors of $u_h^{1},u_h^{2}$ in norms $L_2$ and $H_1$ for 
$\varepsilon_2$  on Kershaw meshes.}\label{tab-14}
	\centering
\begin{tabular}{ccccccc cc }
\hline
nx &  $\|e_h^{1}\|_0$    & $\gamma_0$ &   $\|e_h^{1}\|_1$    & $\gamma_1$ & $\|e_h^{2}\|_0$    & $\gamma_0$ & $\|e_h^{2}\|_1$    & $\gamma_1$\\ \hline
24    & {9.27e-3} &      & {1.54e-1} &      & {2.74e-3} &       &1.48e-1  &\\
48   & { 5.06e-3} & 1.83 & {8.26e-2} & 1.86 & {7.21e-4} & 3.80  &7.29e-2  &  2.15\\
96   & { 2.64e-3} & 1.92 & {4.55e-2} & 1.82 & {1.85e-4} & 3.90  &3.58e-2  & 2.03\\
192   & {1.33e-3} & 1.98 & {2.47e-2} & 1.84 & {4.64e-5} & 3.98  &1.77e-2  & 2.02\\\hline
\end{tabular}}
\end{table}



\begin{table}[H]
{\small
\caption{ The errors of $u_h^{1},u_h^{2}$ in norms $L_2$ and $H_1$ for 
$\varepsilon_3$  on random meshes.}\label{tab-15}
	\centering
\begin{tabular}{ccccccc cc }
\hline
nx &  $\|e_h^{1}\|_0$    & $\gamma_0$ &   $\|e_h^{1}\|_1$    & $\gamma_1$ & $\|e_h^{2}\|_0$    & $\gamma_0$ & $\|e_h^{2}\|_1$    & $\gamma_1$\\ \hline
24    & {1.62e-2} &      & {1.77e-1} &      & {1.60e-3} &       &1.26e-1  &\\
48   & { 8.26e-3} & 1.96 & {1.13e-1} & 1.58 & {9.21e-4} & 1.74  &7.69e-2  & 1.64\\
96   & { 4.26e-3} & 1.94 & {7.65e-2} & 1.48 & {2.38e-4} & 3.87  &3.21e-2  & 2.39\\
192   & {2.15e-3} & 1.98 & {4.22e-2} & 1.81 & {6.63e-5} & 3.59  &1.66e-2  & 1.94 \\\hline
\end{tabular}}
\end{table}

\begin{table}[H]
{\small
\caption{ The errors of$u_h^{1},u_h^{2}$ in norms $L_2$ and $H_1$ for 
$\varepsilon_3$  on Kershaw meshes.}\label{tab-16}
	\centering
\begin{tabular}{ccccccc cc }
\hline
nx &  $\|e_h^{1}\|_0$    & $\gamma_0$ &   $\|e_h^{1}\|_1$    & $\gamma_1$ & $\|e_h^{2}\|_0$    & $\gamma_0$ & $\|e_h^{2}\|_1$    & $\gamma_1$\\ \hline
24    & {1.31e-2} &      & {1.82e-1} &      & {3.42e-3} &       &1.69e-1  &\\
48   & { 7.23e-3} & 1.82 & {1.04e-1} & 1.75 & {8.10e-4} & 4.22  &8.07e-2  & 2.09\\
96   & { 3.83e-3} & 1.89 & {6.18e-2} & 1.68 & {2.02e-4} & 4.01  &3.86e-2  &2.09\\
192   & {1.97e-3} & 1.94 & {3.58e-2} & 1.72 & {4.97e-5} & 4.06  &1.85e-2  &2.08\\\hline
\end{tabular}}
\end{table}


The results of columns from 2 to 5 in Tab. \ref{tab-9} to Tab.\ref{tab-16} show that MFVE-1 scheme is stable and first-order convergent in $L_2$ and $H_1$ norms for convection-dominated cases whether Random meshes or Kershaw meshes,  and the norm $H_1$ of error is approaches first-order as the $\varepsilon$ becomes smaller, i.e., the convection term dominates greatly. 

The results of columns from 6 to 9 in Tab. \ref{tab-9} to Tab.\ref{tab-16} show that MFVE-2 scheme is stable and, respectively, second-order (resp. first-order ) convergent in $L_2$ (resp. $H_1$) norm for convection-dominated cases whether Random meshes or Kershaw meshes,  i.e., MFVE-2 is approximately saturated convergent order as the $\varepsilon$ becomes smaller, i.e., the convection term dominates greatly. 

Hence, MFVE-1 and MFVE-1 schemes are with saturated convergent order, and good at the anisotropic and convection-dominated diffusion problems. The results confirm the correctness of the theoretical conclusions.

\subsection{The positivity-preserving of new schemes}

In the following, we will focus on the positivity-preserving  (monotonicity ) of MFVE-1 and MFVE-2 schemes for convection-dominated diffusion problems.

\begin{example} \label{eg-4}
Consider Problem (\ref{eq-PDE-1}) in the region $\Omega=(0,1)^2$, and choose
\begin{eqnarray*}
\kappa=                                  \left(
                                         \begin{array}{cc}
                                           y^2+ \varepsilon x^2 +\varepsilon  & -(1-\varepsilon) xy \\
                                           -(1-\varepsilon) xy         & x^2+\varepsilon y^2 +\varepsilon \\
                                         \end{array}
                                       \right),
\end{eqnarray*}
where $\varepsilon=5\times 10^{-3}$, and
$$
f(x,y)=\left\{
         \begin{array}{ll}
           1, & \hbox{if~ $(x,y)\in [3/8,5/8]^2$,} \\
           0, & \hbox{otherwise.}
         \end{array}
       \right.
$$
and
$$
v(x,y)= (2x,3y)^T, ~g(x,y)=0,~~(x,y) \in \partial \Omega.
$$
\end{example}

It is solved by MFVE-1 and MFVE-2 schemes. The computational meshes are shown as Fig.\ref{meshes} (a) with the scale $64\times 64$, and the numerical solution obtained by MFVE-2 is shown as Fig. \ref{eg4solution} (a) and (b). From the figure, one can see that the MFVE-2 preserves the positivity of the solution.  The minimum value is 0 and the maximum values $6.54\times 10^{-2}$, which shows that our scheme preserves the positivity of the solution and does not introduce any non-physical oscillations. The results accord with that in the literature \cite{Sheng}. Similar are the results of MFVE-1, and are not listed.

\begin{figure}[h]
     \centering{
    \includegraphics[scale=0.32]{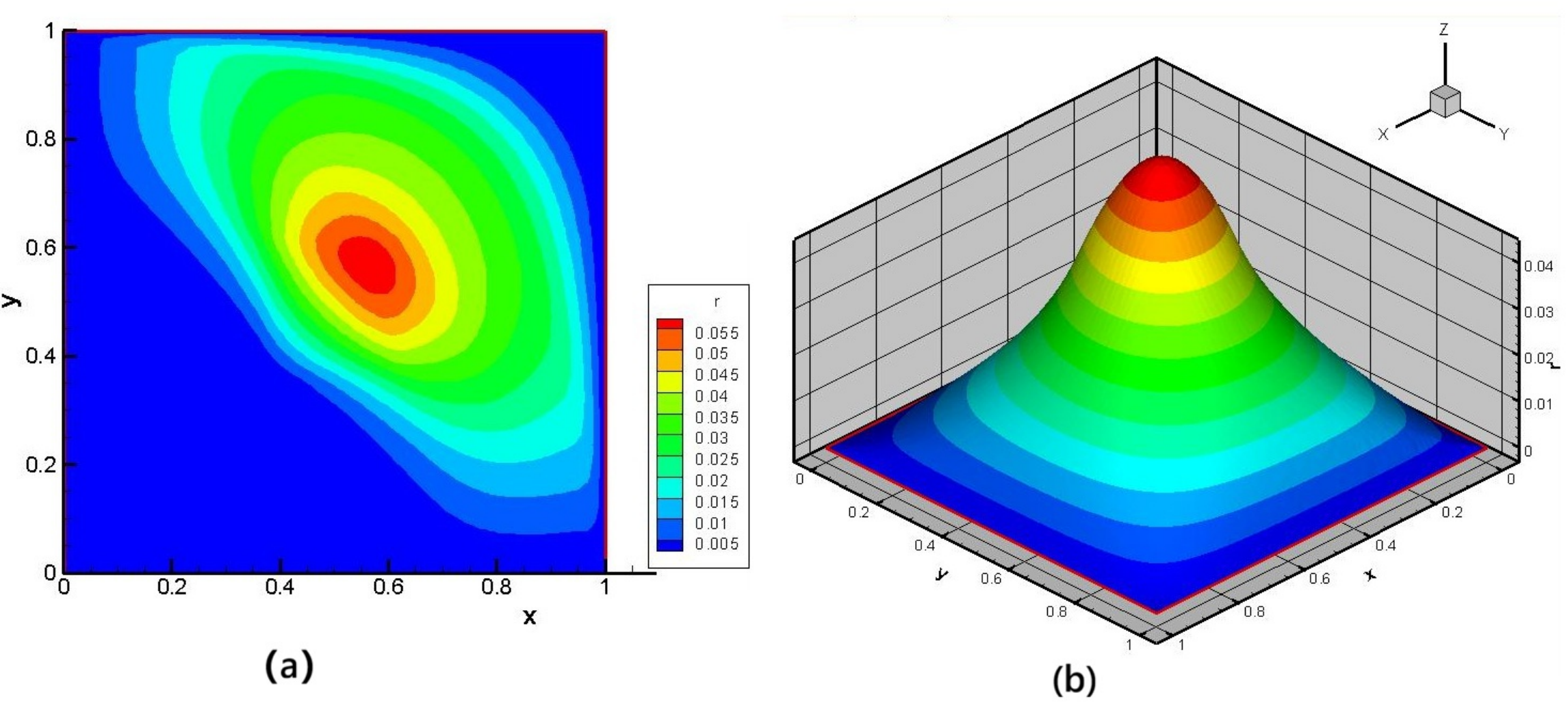}}
    \vskip -0.2cm
    \caption{Verification of the positivity for MFVE-2 scheme on random meshes~: umin =0, umax =$6.54\times 10^{-2}$.}\label{eg4solution}
\end{figure}

\section{Summary}

In this paper, one novel monotonous finite volume element (MFVE) scheme is presented for convection dominant diffusion problems.  
Firstly, one upwind volume is put forward for the discretization of the convection term, which leads to the upwind property of this scheme. 
Secondly, the convection term in the balance equation is directly discretized by numerical integration in upwind volume.
Thirdly, one asymptotic expansion of gradient functions over the dual element of each node is derived. One nonlinear second-order FVE discrete operator for the convection term is constructed by the asymptotic expansion and error estimation of its FVE solution. 
Then, one second-order nonlinear discrete operator for the convection term is constructed by the asymptotic expansion and error estimation of its FVE solution, where some perturbed coefficients are skillfully put into the expansion, which serves as a high-order correction. 
Fourthly, one novel two-order MFVE scheme, also accompanied by another one-order MFVE, is designed for convection-dominated diffusion problems together with one PFVE discrete diffusive operator. The L2 norm of the error of the approximate solution is derived. Finally, numerical results confirm theoretical conclusions.


\vskip 1.5cm
\noindent{\bf{CRediT authorship contribution statement
}}

Cunyun Nie: writing original draft, validation, methodology, investigation, formal analysis; Xiaoling Chen: resources, 
methodology, investigation; Zhujun wang: formal analysis, conceptualization; Zhikun Tian: writing, editor, numerical tests; Chengjie Xia: resources, investigation. 

\vskip 1.5cm


\bibliographystyle{cas-model2-names}

\bibliography{cas-refs}






\end{document}